\documentclass[11pt,english,a4paper]{amsart}

\usepackage{babel}
\usepackage{amsfonts,amssymb,latexsym,xspace,epsfig,graphics,color}  
\usepackage{amsmath,enumerate}  
\numberwithin{equation}{section}  

 \usepackage[latin1]{inputenc}

\title[Adjustement coefficient in some dependent contexts]{Adjustment coefficient for risk processes in some dependent contexts}
\author{ H. Cossette}
\author{E. Marceau}
\author{V. Maume-Deschamps}
\address{Université Laval}
\email{helene.cossette@act.ulaval.ca}
\email{etienne.marceau@act.ulaval.ca}
\address{Université de Lyon, Université Lyon 1, ISFA, laboratoire SAF}

\email{veronique.maume@univ-lyon1.fr}   

\thanks{ \ }

 \newcommand{\eps}{\varepsilon}   
\newcommand{\N}{\ensuremath{\mathbb{N}}}    
\newcommand{\R}{\ensuremath{\mathbb{R}}}     
\newcommand{\Z}{\ensuremath{\mathbb{Z}}}    
\renewcommand{\P}{\ensuremath{\mathbb{P}}}       
\newcommand{\E}{\ensuremath{\mathbb{E}}}

\newcommand{\Id}{\mathbf {1}} 
\newcommand{\ii}{\mathbf{i}}
\newcommand{\jj}{\mathbf{j}}

\newcommand{\cov}{\mbox{\rm Cov}}

\newcommand{\VV}{\mathcal V}
\newcommand{\UU}{\mathcal U}

\newcommand{\LL}{\mathcal L}
 
\newcommand{\NN}{\mathcal N}

\renewcommand{\hat}{\widehat}

\newtheorem{theo}{Theorem}[section]

\newtheorem{prop}[theo]{Proposition}    
\newtheorem{coro}[theo]{Corollary}    
\newtheorem{lem}[theo]{Lemma}

\theoremstyle{definition}
\newtheorem{defi}{Definition}
\newtheorem{ppte}[theo]{Property}

\theoremstyle{remark}
\newtheorem*{rem}{Remark}

\begin{document}
\begin{abstract}
Following \cite{MP}, we study the adjustment coefficient of ruin theory in a context of temporal dependency. We provide a consistent estimator of this coefficient, and perform some simulations.  
\end{abstract}
\maketitle
Adjustment coefficient $w$ for risk processes may describe the behavior of ruin probability. Several results for sums of i.i.d. claims exist: in \cite{G1}, H.U. Gerber gave an exact formula for finite time ruin probabilities involving the adjustment coefficient $w$, \cite{embrechts} provide a consistent estimator of $w$, V. Mammisch \cite{Mam} gave a necessary and sufficient condition for the existence of $w$. In dependent  contexts, let us cite H.U. Gerber \cite{G2} for auto-regressive processes, \cite{estimARMA} for an extension to ARMA processes and \cite{etiennedeptpsmontants,CLM} for the study of the adjustment coefficients in Markovian environments . The main objective of the parper is to provide a non parametric estimation of the adjustement coefficient introduced in \cite{MP} in dependent contexts. We give a general dependent context (weak temporal dependency in the sense of \cite{book}) for which our estimator is consistent. \\
The paper is organized as follows~:
\begin{itemize}
\item Section \ref{setting} contains the definitions and elementary properties of weak-dependent processes as well as adjustment coefficient. To make short, $w^i$, the independent coefficient, will be the adjustment coefficient if the process is i.i.d. while $w^d$ will be the adjustment coefficient of a dependent sequence. 
\item In Section \ref{lim}, we prove that $w^d$ may be seen as a limit (for $r\rightarrow \infty$) of independent coefficients $w^i_r$. We also provide some general examples for which the adjustment coefficient $w^d$ may be defined. 
\item Section \ref{est} is devoted to the estimation of coefficients $w^i$ and $w^d$ and contains the main results~: we construct consistent estimators (see Theorems \ref{TCL2}, \ref{TCL3} and \ref{as_cv}). Note that in \cite{estimARMA}, an estimation of $w^d$ is given for ARMA processes which is based on the estimation of the ARMA parameters. Our procedure is completely non parametric. 
\item In Section \ref{simul} we provide some simulations. 
\end{itemize}
\section{Setting}\label{setting} We consider $(Y_n)_{n\in\N}$ a sequence of random variables and $R_u$ the event $\{Y_n > u \ \mbox{for some} \ n\geq 1\}$. $Y_n$ is interpreted as the value of the claim surplus process of a company at the end of the year $n$. $R_u$ is understood as the ruin event for an intial reserve $u>0$. We could write 
$$Y_n = \sum_{i=0}^n X_i \/ $$
where $X_i$ is the gain/loss of the company during the year $i$. 
\subsection{Weak dependent processes} This last decade, {\em Doukhan and al.} (\cite{book, DD, DL, DP}) have developed new dependence definitions that both extend classical probabilistic definitions (as $\alpha$ or $\Phi$ mixing) and are satisfied for several useful models (like $ARMA$ or $ARCH$) that are neither $\alpha$ nor $\Phi$ mixing in the standard way. Roughly speaking, in the classical probabilistic definitions of mixing, the functions $f$ and $g$ in the definition below (\ref{eqdep}) belong to the whole class of square integrable functions. \\
Define (see \cite{book}) for a real valued or vector valued process $(X_t)_{t\in \N}$,
\begin{equation}\label{eqdep}
\eps(k) = \sup \frac{\left|\cov(f(X_{i_1}\/, \ldots \/, X_{i_u}),g(X_{j_1}, \ldots \/, X_{j_v}))\right|}{c(f\/,g)} 
\end{equation}
where the supremum is taken over multi-indices $\ii=(i_1\/, \ldots \/, i_u)$, $\jj=(j_1\/, \ldots \/, j_v)$ such that:
$$ i_1< \/ \cdots \/ < i_u \leq i_u+k \leq j_1 < \/ \cdots \/ < j_v $$
 and all functions $f~:~\R^u~\longrightarrow~\R$,$g~:~\R^v~\longrightarrow~\R$ are bounded and Lipschitz functions, with respect to the distance~:
$$d(x\/,y) = \sum_{i=1}^p |x_i-y_i| \/, \ x=(x_1\/, \ldots \/, x_p)\/, \ y=(y_1\/, \ldots \/, y_p) \/.$$
\begin{rem}
We could replace the space of Lipschitz functions by other spaces of regular functions (differentiable functions, functions of bounded variation ...), see \cite{estimT} for a general condition of convenient functional spaces. 
\end{rem} 
We define a notion of weak dependence according to the function $c(f\/,g)$.
\begin{defi} \label{def_weak_dep}
Consider the following functions $c(f\/,g)$, defined for \\
$f~:~\R^u~\longrightarrow~\R$ and $g~:~\R^v~\longrightarrow~\R$  bounded and Lipschitz functions, $\mbox{lip}(f)$ is the Lipschitz coefficient of the function $f$. 
\begin{enumerate}
\item $c(f\/,g) = v \Vert f \Vert_\infty \/ \mbox{lip}(g)$, we say that the sequence $(X_i)_{i\in\N}$ is $\theta$-weakly dependent if the corresponding mixing coefficients sequence $(\eps(k))_{k\in\N}$ is summable.
\item $c(f\/,g) = u \mbox{lip}(f) \/ \Vert g \Vert_\infty + v\Vert f \Vert_\infty \/ \mbox{lip}(g)$, we say that the sequence $(X_i)_{i\in\N}$ is $\eta$-weakly dependent if the corresponding mixing coefficients sequence $(\eps(k))_{k\in\N}$ is summable.
\end{enumerate}
\end{defi} 
This class of dependent processes is very rich and enjoy lots of nice probabilistic properties. Let us remark that weak dependent processes need not to be stationary.  \\
For completeness, we recall the definitions of $\alpha$, $\Phi$ and $\Psi $ mixing. $\Psi $ mixing may be defined in a formalism close to that of Definition \ref{def_weak_dep} while for $\alpha$ and $\Phi$ mixing, it is not clear that the same formalism is meaningfull. 
\begin{defi}\label{def_phi}
Consider mixing coefficients $\eps(k)$ defined by Equation (\ref{eqdep}) where the supremum is taken over functions \\
$f~:~\R^u~\longrightarrow~\R$ and $g~:~\R^v~\longrightarrow~\R$ in $L^2$ and  $c(f\/,g) =  \Vert f \Vert_1 \/ \Vert g\Vert_1$. We say that the sequence $(X_i)_{i\in\N}$ is $\Psi$ mixing if the corresponding mixing coefficients $\eps(k)$ are bounded.\\
 $\alpha$ and $\Phi$ mixing coefficients are defined as~:
$$\alpha(\UU\/,\VV) = \sup_{U\in\UU\/, \ V\in\VV} |\P(U\cap V) -\P(U)\P(V) | \/,$$
$$\Phi(\UU\/,\VV) = \sup_{U\in\UU\/, \ V\in\VV} \left|\frac{\P(U\cap V)}{\P(U)} -\P(V)\right| \/.$$
A process $(X_t)_{_\in\Z}$ is $\alpha$ (resp. $\Phi$) mixing is the coeficients
$$\alpha_X(r) = \sup_{i\in\Z}\alpha(\sigma(X_t\/,\ t\leq i)\/, \ \sigma(X_t\/, \ t\geq i+r))\/,$$
$$\mbox{resp.} \ \Phi_X(r) = \sup_{i\in\Z}\Phi(\sigma(X_t\/,\ t\leq i)\/, \ \sigma(X_t\/, \ t\geq i+r))\/$$
go to $0$. 
\end{defi}
\subsection{Adjustment coefficient} In the classical i.i.d. (i.e. the $X_i$ are i.i.d. random variables) model of ruin theory, the adjustment coefficient $w>0$ is defined as the unique positive solution of $\lambda(w) = 0$ with  
$$\lambda(t) = \log \E\left[\exp(t X_1)\right] $$ 
assumed to be well defined. Mammischt (\cite{Mam})  gave a necessary and sufficient existence condition for $w$. The importance of the adjustment coefficient is revealed by the exact formula due to Gerber (\cite{G1})~: let $T$ be the ruin time ($T = \inf\{k \in \N \ / \ Y_k >u \}$),
$$\P(R_u) = \P(T<\infty) =\frac{e^{-wu}}{\E[e^{-wY_T}|T<\infty]} \/,$$
and the famous de Finetti bound follows~: 
$$\P(R_u)\leq e^{-wu}\/.$$ 
We shall focus on the following  asymptotic result also due to Lundberg~:
$$\lim_{u\rightarrow \infty}  \frac{\log\P(R_u)}u = -w \/.$$
As already mentioned above,  several attempts to extend these results to dependent and/or non stationary settings have been proposed. We wish to give a general dependent setting in which such an asymptotic result holds, as well as provide a consistent estimator to the adjustment coefficient in this dependent context. Our approach does not require a precise knowledge of the dependence structure, nor on the law of $(X_i)_{i\in\N}$ but only an information on the speed of mixing (given by Equation (\ref{eqdep})). \\
Following \cite{MP} we assume  that~: there exists $t_0>0$ such that for all $0<t<t_0$,  
\begin{equation}\label{limit}
c(t) = \lim_{n\rightarrow \infty} \frac{\log \E\left[\exp(tY_n)\right]}n \  \mbox{exists.} \end{equation}
Also, there exists $0<t<t_0$ such that $c(t)=0$. \\
\ \\
We shall provide a sufficient condition that implies existence and uniqueness of a positive solution to $c(t)=0$, provided that (\ref{limit}) is satisfied. We shall denote $w^d$ this unique solution. We shall also denote by $w^i$ the unique positive solution to $\lambda(t)=0$ with 
$$\lambda(t) = \log  \E\left[\exp(t X_1)\right]\/. $$
Of course, if the sequence $X_i$ is i.i.d. then $w^i=w^d$. 
\subsection{Existence condition}\label{sec_exist} 
In Mammisch (\cite{Mam}) it is proven that $w^i$ exists (and is unique as a positive solution to $\lambda(t)=0$) if and only if the following three conditions are satisfied. We shall denote by (E) these three conditions. 
\begin{enumerate}
\item $\E(X_1)<0$, 
\item $\P(X_1>0)>0$ and 
\item either $a<\infty$ and $\E(e^{aX_1})\geq 1$ or $a=\infty$ where 
\begin{equation}\label{def_a}a = \sup\{t\geq 0 \/, \ \E(e^{tX_1})<\infty \} \/.\end{equation}
\end{enumerate}
This condition, together with the weak dependence assumption is sufficient to get consistency and asymptotic normality of the estimator of $w^i$. In order to get existence and uniqueness (as a positive solution to $c(u)=0$) of $w^d$, we shall need some additional conditions. 
\begin{ppte}\label{exist}
Assume that the limit $c(u)$ is well defined on $[0\/,u_0[$, in particular, for all $t<u_0$, for $n$ large enough, $\E(e^{tY_n})<\infty$. Assume that (E) is satisfied for $0<a\leq u_0$, and
\begin{enumerate}
\item $c_n(t)$ exists for all $n$ and $0\leq t < a$,
\item for large enough $n$, $\P(Y_n>0)>0$
\item if $a<\infty$ then for $n$ large enough, $\displaystyle \lim_{t\rightarrow a-} \E(e^{tY_n}) \geq 1$,
\item $c'(0+)<0$ or equivalently, $\exists t>0$ such that $c(t)<0$ 
\end{enumerate}
 then there exists a unique positive solution to $c(u)=0$. This solution is denoted by $w^d$. 
\end{ppte} 
\begin{proof}
We adapt Mammish's arguments. Recall that any convex function is continuous and admits left and right derivatives on all points where it is defined. Moreover, if $f_n$ is a sequence of convex functions defined on $[0\/,u_0[$ and converging to $f$ on $[0\/,u_0[$ then $f$ is a convex function and the convergence is uniform on any compact subset of $[0\/,u_0[$.\\
The function $c$ is the limit on $[0\/,u_0[$ of convex functions 
$$c_n(t) = \frac1n \log \E(e^{tY_n}) \/.$$
Thus $c$ is a convex function with $c(0)=0$ and we assume that $c'(0+)<0$ (which is equivalent to $\exists \ t>0$ such that $c(t)<0$ by convexity). 
\begin{itemize}
\item If $a<\infty$ then, we assume  that for large enough $n$, $\infty\geq \E(e^{aY_n}) \geq 1$ thus $\infty \geq c(a) \geq 0$ . Since $c$ is continuous (because it is convex), we deduce that there exists $w>0$ such that $c(w)=0$. This solution is unique because of the convexity of $c$. 
\item If $a=+\infty$, then $c(t)$ is well defined for any $t\in\R^+$. Because  $\P(Y_n>0)>0$ for $n$ large enough, we have that for large enough $n$, $\displaystyle\lim_{t\rightarrow \infty}\E(e^{tY_n})=+\infty$. As a consequence, we have that $\displaystyle\lim_{t\rightarrow\infty}c(t)\geq 0$. Since we assume that $c(t)<0$ for some $t$, the convexity of $c$ then implies that there exists $t_0$ such that $c'(t\pm)>0$ for any $t >t_0$. We deduce that $c(t)>0$ for $t$ large enough.  Since $c$ is continuous (because it is convex), we deduce that there exists $w>0$ such that $c(w)=0$. This solution is unique because of the convexity of $c$.
\end{itemize}
\end{proof}
\begin{rem}
The condition $\P(Y_n>0)>0$ is necessary because if it exists $w>0$ such that $c(w)=0$ then by convexity, either there exists $t>0$ such that $c(t)>0$ or $c(t)=+\infty$ for all $t\geq w$ which implies that $\E(e^{tY_n})>1$ for large enough $n$ which implies $\P(Y_n>0)>0$ for large enough $n$.  
\end{rem}
\section{Limit result and examples}\label{lim}
\subsection{Asymptotic behavior for ruin probability}
In \cite{MP}, it is proven that if the adjustment coefficient $w^d$ exists then it describes the asymptotic behavior of the ruin  probability in the following sense:
\begin{equation}\label{asymp}
\lim_{u\rightarrow \infty}\frac{\log \P(R_u)}u =-w^d \/.
\end{equation}
As a consequence to (\ref{asymp}), we obtain that if it exists, $w^d$ is the limit of the adjustment coefficients of the sequence $(Y_n)_{n\in\N}$.
\begin{coro}\label{conv_adjust}
Assume that the hypotheses of Property \ref{exist} are satisfied. For large enough $n$, there exists a unique $w_n>0$ such that 
$$\E(e^{w_nY_n})=1 $$
and $w^d = \displaystyle \lim_{n\rightarrow \infty} w_n$. 
\end{coro}
\begin{proof}
The existence of $w_n$ follows from \cite{Mam}: $Y_n$ satisfies the existence hypotheses of Mammish for $n$ large enough. Applying Markov's inequality, we get for all $K>0$,
$$w_n \leq \frac{-\log \P(Y_n>u)}u \/.$$
Then, (\ref{asymp}) implies that any limit point $w$ of the sequence $w_n$ verifies~: $w \leq w^d$. The convergence of functions $c_n$ to $c$ is uniform on $[0\/, w^d]$, thus we have that $c(w)=0$, so that either $w=0$ or $w=w^d$. Now, $0$ cannot be a limit point of the sequence $w_n$ because otherwise we would have that $c(t)\geq 0$ for all $t\geq 0$ which  contradicts the hypotheses of  Property \ref{exist}. We conclude that $w_n \displaystyle\stackrel{n\rightarrow \infty}{\longrightarrow} w^d$.
\end{proof}
Let us give some examples for which the function $c(t)$ is well defined. We recall the following result on approximate sub additive  sequences due to Hammersley \cite{approx_sub}.
\begin{lem}\label{approx}
Assume $h~:~\N~\longrightarrow~\R$ be such that for all $n\/, m\geq 1$, 
$$h(n+m)\leq h(n)+h(m)+\Delta(m+n)\/,$$
with $\Delta$ a non decreasing sequence satisfying~:
\begin{equation}\label{eq:approx}
\sum_{r=1}^\infty \frac{\Delta(r)}{r(r+1)} <\infty\/.
\end{equation}
Then, $\lambda=\displaystyle\lim_{n\rightarrow \infty} \frac{h(n)}n$ exists and is finite. Moreover, for all $m\geq 1$,
$$\lambda \leq \frac{h(m)}m-\frac{\Delta(m)}m+4\sum_{r=2m}^\infty \frac{\Delta(r)}{r(r+1)} \/.$$
\end{lem}
Of course, for $\Delta(r)= O(1)$, then (\ref{eq:approx}) is satisfied. Lemma \ref{approx} asserts that $\Delta(r)$ may go to infinity but not too fast.
\subsection{$\Psi$-mixing processes} According to Definition \ref{def_phi}, we consider the following classical $\Psi$-mixing condition~: 
\begin{equation}\label{eq_phi} 
\Psi(k) = \sup \frac{\cov(f(X_{i_1}\/,\ldots \/, X_{i_u})\/,g(X_{j_1}\/, \ldots \/, X_{j_v}) )}{\Vert f \Vert_1 \/ \Vert g\Vert_1}  < \infty\/, 
\end{equation} 
where the supremum is taken over functions $f\/, \ g\in L^2$ and over multi-indices $\ii = (i_1\/,\ldots \/, i_u)$ and $\jj=(j_1\/, \ldots \/, j_v)$  with $i_1<\cdots < i_u < i_u+k\leq j_1 < \cdots <j_v$.  The sequence if $\Psi$-mixing is bounded. 
\begin{prop} 
Assume that for  $t\in[0\/,a[$, for all $n\in \N$, $\E(e^{2tS_n})<\infty$ and $(X_n)_{n\in\N}$ is a  $\Psi$-mixing process then  
$$\lim_{n\rightarrow  \infty}\frac1n\log\E(e^{tS_n})$$ 
exists for any $t \in[0\/,a[$.
\end{prop} 
\begin{proof} 
Using (\ref{eq_phi}),  we have~:
$$\left|\E(e^{tS_{n+m}}) -\E(e^{tS_n})\E(e^{tS_m}) \right| \leq \Psi(1) \E(e^{tS_n})\E(e^{tS_m}) \/.$$
We conclude the proof by  using Lemma \ref{approx}. 
\end{proof} 
Examples of  $\Psi$-mixing processes are finite state Markov chains of any order but also Variable Length Markov Chains (VLMC) on a finite state (see Lemma 3.1 in \cite{context}).\\ 
Even if  $\Psi$ ($\Phi$, $\alpha$)-mixing processes are often used in probability theory, lots of useful processes (like ARMA processes) are not $\Psi$ ($\Phi$, $\alpha$)-mixing. In the following two subsections, we provide a class of $\eta$ mixing processes for which the function $c(t)$ is well defined provided the mixing is sufficiently fast and the variables $X_i$ are almost surely bounded. This condition is close to the one used in \cite{Bric_dembo} by Bric and Dembo for one other class of mixing processes (namely $\alpha$ mixing processes). Then we prove that if the sequence has some structure (here, the sequence is a Bernoulli shift)  then the condition on the speed of mixing may be weakened. We refer to \cite{book} for other examples of $\theta$ and $\eta$ weakly mixing processes (including ARMA and ARCH processes). 
\subsection{Super mixing  processes} We prove that if the process $(X_i)_{i\in\N}$ is $\eta$ weakly dependent (recall Definition \ref{def_weak_dep}) with dependence coefficient  $\eps(n)=O\left(\theta^{n(\ln n)^\beta}\right)$ with $0<\theta <1$ and $\beta >1$ then the function $c(t)$ is well defined provided $|X_i|\leq M $a.e. A $\eta$ weakly dependent process with dependence coefficient  $\eps(n)=O\left(\theta^{n(\ln n)^\beta}\right)$will be called a super mixing process.
\begin{prop} 
Assume $(X_n)_{n\in\N}$ is a $\eta$ weakly dependent process with mixing coefficients $\eps(n) = O(e^{-c n(\ln n)^\beta})$ with $c>0$, $\beta>1$. Moreover, assume that there exists $M>0$ such that $|X_i|\leq M $ a.e. Then the sequence $c(t)$ is well defined on $\R$. 
\end{prop}  
\begin{proof} 
For any $0<j<k$,  let $\displaystyle S_j^k = \sum_{\ell=j}^k X_\ell$, because $|X_i|\leq M$ a.e.,
$$e^{-t(k-j)M}\leq  \E(e^{tS_j^k}) \leq e^{t(k-j)M}\/.$$
Also, for any $j\leq  \ell \leq j$,
$$e^{-t(\ell-j) M}  \E(e^{tS_{\ell}^k})\leq \E(e^{tS_j^k}) \leq e^{t(\ell - j)M} \E(e^{tS_\ell^k}) \/.$$
Remark that the function~:  $x \mapsto e^{tx}$ is bounded above by $e^{tM}$ and has Lipschitz constant $te^{tM}$ for $x\in[-M\/,M]$. \\
Fix an integer $0<r< \max(n\/,m)$.  Firstly, assume that $n\leq m$ and, using the definition of $\eta$ weak dependence, we get:
\begin{eqnarray*}
\lefteqn{\E(e^{tS_{n+m}}) = \E(e^{tS_1^n}e^{tS_{n+1}^{n+m}})}\\
&=&\E\left(e^{tS_1^{n}} \cdot e^{tS_{n+r+1}^{n+m}} \cdot e^{tS_{n+1}^{n+r}}\right)\\
&\leq&e^{trM}\left(\E(e^{tS_1^{n}})\cdot \E(e^{tS_{n+r+1}^{n+m}}) +  \eps(r) (n+m) te^{tnM}e^{t(m-r)M}  \right)\\
&\leq&e^{2trM} \E(e^{tS_n})\E(e^{tS_m}) + (n+m)\eps(r)te^{2t(n+m)M} \E(e^{tS_n})\E(e^{tS_m})\\
&\leq& \E(e^{tS_n})\E(e^{tS_m}) (e^{2trM}+(n+m) \eps(r)te^{2t(n+m)M})\/.
\end{eqnarray*}
We conclude the proof by choosing $r=O(\frac{n+m}{\ln(n+m)^\kappa})$, $1<\kappa<\beta$ and applying Lemma \ref{approx} with a function:
$$\Delta(r) = O\left(\frac{r}{(\ln r)^\kappa}\right)\/.$$
If $n>m$, the proof is similar but uses the decomposition~:\\
$S_{n+m} = S_1^m+S_{m+1}^{m+r}+S_{m+r+1}^{n+m} \/.$
\end{proof}
\subsection{Bernoulli shifts}
Causal Bernoulli shifts are processes defined as:
$$X_n  = H(\xi_{n-j}\/, \ j\in \N)$$
with $H$ a measurable function and $(\xi_n)_{n\in\Z}$ an i.i.d. process. We shall assume the following regularity condition on $H$~: define the continuity coefficients
$$d_n = \Vert \sup_{u=(u_0\/, u_{-1}\/,\ldots)} |H(\xi_{n-i}\/, i\in\N) - H(\xi_n\/,\ldots \/, \xi_1\/, u_0\/, u_{-1}\/, \ldots)| \/ \Vert_\infty$$
and  assume that the sequence $d_n$ is summable. By adapting the arguments of \cite{book} we may prove that such a process is $\theta$-dependent with mixing coefficient $\theta(n)=d_n$. 
\begin{prop}\label{Bern_shift}
Assume that $(X_n)_{n\in\N}$ is a Bernoulli shift satisfying the summability condition for the continuity coefficients $d_n$. Then the sequence $\ln \E(e^{tS_n})$  satisfies the hypotheses of Lemma \ref{approx} and thus $c(t)$ is well defined.
\end{prop}
\begin{proof}
We fix a sequence of real numbers $u=(u_0\/, u_{-1}\/, \ldots )$ and we write
\begin{eqnarray}
S_i^\ell&=&\sum_{j=i}^\ell X_j \nonumber \\
	&=&\sum_{j=i}^\ell H(\xi_j\/, \ldots \/, \xi_i\/,\ldots )\nonumber\\
	&=&\underbrace{\sum_{j=i}^\ell H(\xi_j\/, \ldots \/, \xi_i\/, u_0\/, u_{-1}\/, \ldots )}_{:=U_i^\ell} +\nonumber \\
&& \underbrace{\sum_{j=i}^\ell H(\xi_j\/, \ldots \/, \xi_i\/,\ldots ) -  H(\xi_j\/, \ldots \/, \xi_i\/, u_0\/, u_{-1}\/, \ldots )}_{:=d_i^\ell(u\/,\xi)} \/.\label{SetU}
\end{eqnarray}
Using the stationarity of $(\xi_n)_{n\in\N}$, we have that 
$$d_i^\ell \leq \sum_{j=i}^\ell d_{j-i} = \sum_{j=0}^{\ell-i} d_j \/.$$
Now, $U_1^n$ and $U_{n+1}^m$ are independent random variables and thus,
$$\E(e^{tS_{n+m}}) \leq e^{t(\sum_{i=1}^nd_i +\sum_{i=1}^m d_i) }\E(e^{tU_1^n})\E(e^{tU_{n+1}^{n+m}})\/.$$
Applying once more (\ref{SetU}), we get:
$$\E(e^{tS_{n+m}})\leq \exp\left[4t\displaystyle\sum_{i=1}^{\max(n\/,m)}d_i\right]\E(e^{tS_n})\E(e^{tS_m})\/.$$
If we denote $D = \displaystyle \sum_{j\in\N}d_j$, we get
$$\E(e^{tS_{n+m}}) \leq e^{4tD} \E(e^{tS_n})\E(e^{tS_m}) $$
and we conclude by applying Lemma \ref{approx}. 
\end{proof}
\begin{rem}
In the above proposition, we could replace the hypotheses of summability of the sequence $(d_j)_{j\in\N}$ by the summability of 
$$\frac{\Delta(r)}{r(r+1)} \ \mbox{with} \ \Delta(r) = \sum_{j=0}^r d_j \/,$$
Lemma \ref{approx} would apply as well and the limit $c(t)$ would be well defined. 
\end{rem}
In \cite{book}, it is mentionned that stationary ARMA processes are  examples of Bernouilli shifts. Non linear autoregressive processes may also be examples of Bernouilli shifts. Nevertheless, in order to satisfy that the $d_j$ are well defined, it requires that the innovation is bounded. We claim that Proposition \ref{Bern_shift}  remains true for some unbounded Bernouilli shifts but we where unable to prove it. 
\section{Estimation}\label{est}
\subsection{Definition of estimators}
In this section, we assume that the sequence $(X_i)_{i\geq 1}$ is stationary and that the hypotheses of Proposition \ref{exist} are satisfied. For $r\in\N$,   the functions $\E(e^{tX_1})$ and $\E(e^{tY_r})$ may be estimated by their empirical moment versions: for $k\in \N$,
$$\hat{m}_k(t) = \frac1k \sum_{i=1}^k e^{t X_i}\/,$$
$$\hat{M^r_k}(t) = \frac1k\sum_{i=0}^{k-1} e^{tZ_i^r}\/,$$
where $Z_i^r = \displaystyle \sum_{j=1}^{r} X_{j+ir}$. Then we define $\widehat{w}^i$ as the  positive solution to $\log \hat{m}_k(t) =0$ and $\widehat{w}_r$ as the positive solution to $\frac1r \log \hat{M^r_k}(t) = 0$. We shall prove that  $\widehat{w}^i$ is a consistent estimator of $w^i$ and there exists an $r=r(k)$ such that $\widehat{w}_r=\widehat{w}^d$ is a consistent estimator of  $w^d$. We shall also prove that they satisfy a central limit theorem.\\
Before stating and proving our main results on the asymptotic properties of the estimators $\widehat{w}^i$ and $\widehat{w}_r$, we prove that $\widehat{w}^i$ exists almost surely (the proof for $\widehat{w}_r$ will be done later because it requires some weak dependence property for $(Z_i^r)_{i\in\N}$ which is proven in Lemma \ref{heredit_eta}).
\begin{prop}\label{exist_w^i}
Assume that the sequence $(X_i)_{i\in\N}$ is $\eta$ or $\theta$ weakly dependent and satisfies Condition (E). Then, $\widehat{w}^i$  exists eventually almost surely as $k\longrightarrow \infty$.
\end{prop}
\begin{proof}
We begin by noting that the $\theta$ or $\eta$ weak dependence implies ergodicity, see for example \cite{DD}. Following Section \ref{sec_exist}, we have that $\widehat{w}^i$ exists and is unique if and only if,
\begin{enumerate}
\item $\frac1k \sum_{i=1}^k X_i <0$,
\item $\{i=1\/, \ldots \/, k \ / \ X_i>0\}$ is not empty.
\end{enumerate} 
Mammisch's third condition is satisfied with $a=\infty$ because here, expectations are finite sums. The two above conditions are eventually almost everywhere satisfied because of the ergodic theorem.
\end{proof}
\subsection{Asymptotic properties of $\widehat{w}^i$}
Asymptotic properties of the estimators $\widehat{w}^i$ and $\widehat{w_r}$ are done by using the same approach as the one used to prove results on asymptotic properties for $M$-estimators in a parametric context (see \cite{VanDerVaart} section $5$). 
It is known (see \cite{DL,book}) that the process $(X_i)_{i\in\N}$ satisfies a central limit theorem with asymptotic variance 
$$\Gamma^2=\sum_{i=0}^{\infty}\cov(X_0\/,X_i)\/,$$
provided that the sequence $\cov(X_0\/,X_i)$ is summable.\\
We obtain a central limit theorem for $\widehat{m}_k(t)$ by  proving that the sequence $(e^{tX_n})_{n\in\N}$ is also weakly dependent. Then we prove a central limit theorem for the $M$-estimator $\widehat{w}^i$. \\
Let us recall the following results  from \cite{book} (Theorem 7.1 and Section 7.5.4).
\begin{theo}\label{TCL1}
Let  $(Z_n)_{n\in\N}$ be an $\eta$-weakly dependent sequence with  $\eps(n) = O(n^{-2-\kappa})$ for $\kappa>0$. Then,
$$\Gamma^2 = \sum_{n\geq 0} \cov(Z_0\/, Z_n) \/,$$
is well defined  and
$$\sqrt{n}\left(\frac1n\sum_{i=0}^{n-1}Z_i - \E(Z_0) \right) \stackrel{n\rightarrow \infty}{\longrightarrow} \NN(0\/,\Gamma^2)\/.$$
\end{theo} 
As a consequence, we get the following consistency result for $\widehat{w}^i$ as well as asymptotic normality. In order to use Theorem \ref{TCL1}, we first need to prove that the sequence $(e^{tX_n})_{n\in\N}$ is also $\eta$-weakly dependent with $\eps(n)$ decreasing to zero sufficiently rapidly. 
\begin{theo}\label{TCL2}
Assume $(X_n)_{n\in\N}$ is $\eta$-weakly dependent with  $\eps(n) = O(\theta^n)$, $0<\theta<1$.  We have for any $t\in[0\/,u_0[$ that
$$\Gamma^2(t) = \sum_{n\geq 0} \cov(e^{tX_0}\/,e^{tX_n})$$
is well defined on $[0\/,u_0[$ and
$$\sqrt{n}\left(\widehat{m}_n(t) - \E(e^{tX_0}) \right) \stackrel{k\rightarrow \infty}{\longrightarrow} \NN(0\/,\Gamma^2(t))\/.$$
$\widehat{w}^i$ converges in  probability to $w^i$ and 
$$\sqrt{n} \left(\widehat{w}^i-w^i\right)  \stackrel{n\rightarrow \infty}{\longrightarrow} \NN(0\/,\Gamma^2_i)$$
with $\Gamma^2_i= \displaystyle\frac{\Gamma^2(w^i)}{\E(X_1e^{w^iX_1})^2}$. 
\end{theo}
As already mentioned, we begin by proving that the sequence of random variables $(e^{tX_n})_{n\in\N}$ is  $\eta$-weakly dependent. 
\begin{lem}\label{heredit_eta}
Assume $(X_n)_{n\in\N}$ is $\eta$-weakly dependent with mixing coefficient $\eps(r)$. Then, for any $t\in[0\/,u_0[$, the  sequence of random variables $(e^{tX_n})_{n\in\N}$ is  $\eta$-weakly dependent with mixing coefficient $\eps_t(r) \leq 2\E(e^{(t+\kappa)X_1}) \displaystyle\eps(r)^{\frac{\kappa}{t+\kappa}}$ with  $\kappa>0$ such that $t+\kappa\in[0\/,u_0[$.
\end{lem}
\begin{proof}
We follow the proof of Proposition 2.1 in \cite{book}. Let $f$ and $g$ be two Lipschitz functions and for $M>0$ fixed, $x\in\R$, denote $x^{(M)}=\min(x\/,M)$. Assume $(\ii\/,\jj)$ are multi-indices such that 
$$i_1<\cdots < i_u\leq i_u+r\leq j_1<\cdots < j_v \/,$$ 
and define:
\begin{eqnarray*}
F(X_\ii) = f(e^{tX_{i_1}}\/, \ldots \/, e^{tX_{i_u}}) && F^{(M)}(X_\ii) = (e^{tX^{(M)}_{i_1}}\/, \ldots \/, e^{tX^{(M)}_{i_u}}) \/,\\ 
G(X_\jj) = g(e^{tX_{j_1}}\/, \ldots \/, e^{tX_{j_v}}) && G^{(M)}(X_\jj) = (e^{tX^{(M)}_{j_1}}\/, \ldots \/, e^{tX^{(M)}_{j_v}}) \/.
\end{eqnarray*}
Then,
\begin{eqnarray*}
|\cov(F(X_\ii)\/,G(X_\jj))| & \leq & 2\Vert f \Vert_\infty \E(|G(X_\jj)-G^{(M)}(X_\jj)|) \\
	& & +2\Vert g \Vert_\infty \E(|F(X_\ii)-F^{(M)}(X_\ii)|) \\
	& & + |\cov(F^{(M)}(X_\ii)\/, G^{(M)}(X_\jj))|\/.
\end{eqnarray*}
Recall that $\E(e^{sX_i})<\infty$ for any $s\in[0\/,u_0[$ and let $\kappa>0$ be such that 
$$t+\kappa \in [0\/,u_0[\/,$$
then  $\E(|e^{(t+\kappa)X_i}|)<\infty$, using the Markov inequality, we get:
\begin{eqnarray*}
 \E(|G(X_\jj)-G^{(M)}(X_\jj)|) &\leq& \mbox{lip}(g) \sum_{k=1}^v\E(|e^{tX_{j_k}}-e^{tX^{(M)}_{j_k}}|)\\
&\leq& 2v \/ \mbox{lip}(g) e^{-\kappa M} \E(e^{(t+\kappa)X_1})\/.
\end{eqnarray*}
Also, since $(X_n)_{n\in\N}$ is $\eta$-weakly dependent,
$$|\cov(F^{(M)}(X_\ii)\/,G^{(M)}(X_\jj))|\leq (u\/\mbox{lip}(f)\ \Vert g\Vert_\infty+ v \/ \mbox{lip}(g)\Vert f\Vert_\infty)te^{tM}\eps(r) \/.$$
Finally, we obtain for $M\geq 1$,
\begin{eqnarray*}
|\cov(F(X_\ii)\/,G(X_\jj))| & \leq & (u\/\mbox{lip}(f)\ \Vert g\Vert_\infty+ v \/ \mbox{lip}(g)\Vert f\Vert_\infty)\\
& & \times (te^{tM}\eps(r) +2\E(e^{(t+\kappa)X_1}) e^{-\kappa M} )\/.
\end{eqnarray*}
To conclude, we choose $M= -\frac1{t+\kappa} \ln(\eps(r))$.
\end{proof}
\begin{proof}[Proof of Theorem \ref{TCL2}]
Lemma \ref{heredit_eta} together with Theorem \ref{TCL1} imply that for any $t \in[0\/,u_0[$,
$$\sqrt{n}\left(\widehat{m}_n(t) - \E(e^{tX_0}) \right) \stackrel{n\rightarrow \infty}{\longrightarrow} \NN(0\/,\Gamma^2(t))$$
provided that the sequence $\eps_t(r)=\displaystyle\eps(r)^{\frac{\kappa}{t+\kappa}} = O(r^{-2-\alpha})$ for some $\alpha>0$. Since we assume that $\eps(r)=O(\theta^r)$, $0<\theta<1$, this condition is satisfied. \\
This convergence in law also leads to 
$$\widehat{m}_n(t) \stackrel{n\rightarrow \infty}{\longrightarrow} \E(e^{tX_0}) \ \mbox{in probability} \/,$$
moreover we have that this convergence takes place almost everywhere because of the ergodic theorem. \\
Now, let us consider the estimator $\widehat{w}^i$ of $w^i$. Following the proof of Lemma 5.10 in \cite{VanDerVaart}, we have that $\widehat{w}^i$ converges to $w^i$ in probability (this uses the convergence in probability of $\widehat{m}_k(t)$, the continuity of the map $t\mapsto \widehat{m}_k(t)$ and the uniqueness of $\widehat{w}^i$ as a positive solution to $ \widehat{m}_k(t)=1$). The central limit theorem follows now from the $\Delta$ method~:
$$\widehat{m}_k(\widehat{w}^i)-\widehat{m}_k(w^i) = (\widehat{w}^i-w^i) \frac{\partial \widehat{m}_k(w^i)}{\partial t} + \frac12 (\widehat{w}^i-w^i)^2 \frac{\partial^2 \widehat{m}_k(\widetilde{w})}{\partial t^2} \/,$$
with $\widetilde{w} \in [\min(w^i\/, \widehat{w}^i)\/, \max(w^i\/, \widehat{w}^i)]$. Thus,
$$\sqrt{k}(\widehat{w}^i-w^i) = \frac{\sqrt{k}(\widehat{m}_k(\widehat{w}^i)-\widehat{m}_k(w^i))}{\displaystyle \frac{\partial \widehat{m}_k(w^i)}{\partial t} + \frac12 (\widehat{w}^i-w^i) \frac{\partial^2 \widehat{m}_k(\widetilde{w})}{\partial t^2}}\/.$$
 We have that 
\begin{eqnarray*}\sqrt{k}(\widehat{m}_k(\widehat{w}^i)-\widehat{m}_k(w^i)) & =&\sqrt{k}( 1-\widehat{m}_k(w^i))\\
&=& \sqrt{k}(\E(e^{w^iX_1}) -\widehat{m}_k(w^i))
\end{eqnarray*}
and therefore it is asymptotically normal with zero mean and variance $\Gamma^2(w^i)$. Moreover,
$$\frac{\partial \widehat{m}_k(w^i)}{\partial t}= \frac1k \sum_{j=1}^{k} X_je^{w^iX_j} \/.$$
This quantity converges in probability to $\E(X_1e^{w^i X_1})$ and we have that\\ $\displaystyle (\widehat{w}^i-w^i) \frac{\partial^2 \widehat{m}_k(\widetilde{w})}{\partial t^2}$ goes to zero in probability. Finally, we have proven that
$$\sqrt{n}(\widehat{w}^i-w^i)\stackrel{k\rightarrow \infty}{\longrightarrow} \NN\left(0\/, \displaystyle\frac{\Gamma^2(w^i)}{\E(X_1e^{w^iX_1})^2}\right)\/.$$
\end{proof}
\begin{rem}
We could relax the hypotheses that $\eps(r)$ decreases to $0$ exponentially fast. It is sufficient that $\displaystyle\eps(r)^{\frac{\kappa}{t+\kappa}} = O(r^{-2-\alpha})$ for some $\alpha>0$. For example, some intermediate speed of mixing like $\eps(r) = \displaystyle O\left(e^{-K(\ln r)^\beta}\right)$, $K>0$, $\beta>1$ or $\eps(r) = \displaystyle O\left(\theta^{n^\alpha}\right)$, $0<\theta<1$, $0<\alpha < 1$ is convenient.
\end{rem}
\subsection{Asymptotic properties of $\widehat{w}^d$}
Now, we are interested in the consistency of $\widehat{w}^d$. 
\begin{theo}\label{TCL3}
Assume that hypotheses of Theorem \ref{TCL2} are satisfied.  Then there exists a sequence $r=r(k) \stackrel{k\rightarrow \infty}\longrightarrow \infty$ such that  $\widehat{w}^d$ converges in probability to $w^d$. 
\end{theo}
Theorem \ref{TCL3} will be proven by rewriting \ref{TCL2} for $\widehat{w}_r$ instead of $\widehat{w}^i$ and then by using Corollary \ref{conv_adjust}. We only need to prove that the sequence $(Z_i^r)_{i\in\N}$ satisfies some weak dependence property. 
\begin{lem}\label{weak_dep_Z}
Assume that $(X_n)_{n\in\N}$ is $\eta$-weakly dependent with mixing coefficient $\eps(k)$. Then, the sequence $(Z_i^r)_{i\in\N}$ is $\eta$-weakly dependent with mixing coefficient $\eps_Z(k) = r\eps(r(k-1))$. 
\end{lem}
\begin{proof}
 Let $f$ and $g$ be two Lipschitz functions. Assume $(\ii\/,\jj)$ are multi-indices such that 
$$i_1<\cdots < i_u\leq i_u+k\leq j_1<\cdots < j_v \/.$$ 
\begin{eqnarray*}
Then, 
\lefteqn{\cov(f(Z_{i_1}^r\/, \ldots \/,Z_{i_u}^r)\/, g( Z_{j_1}^r\/, \ldots \/,Z_{j_v}^r))}\\
&=& \cov\left(\widetilde{f}(X_{i_1r+1}\/, \ldots \/, X_{i_1(r+1)}\/, X_{i_2r+1}\/, \ldots \/, X_{(i_u+1)r})\right. \/, \\
&& \left. \widetilde{g}(X_{j_1r+1}\/, \ldots \/, X_{j_1(r+1)}\/, X_{j_2r+1}\/, \ldots \/, X_{(j_v+1)r})\right)\\
&\leq& r \eps((k-1)r) \left(u\mbox{lip}f \Vert g\Vert_\infty +v \Vert f \Vert_\infty \mbox{lip} f \right) \/,
\end{eqnarray*}
where $\widetilde{\varphi}(x_1\/, \ldots \/, x_r\/, \ldots \/, x_{rk}) = \displaystyle\varphi\left(\sum_{i=1}^{r}x_i\/, \ldots \/, \sum_{i=1}^r x_{r(k-1)+i}\right)$.
\end{proof}
As a corollary to Lemma \ref{weak_dep_Z}, we deduce that for any $r$, $\widehat{w}_r$ exists eventually almost surely. 
\begin{coro}
Assume that the hypotheses of Property \ref{exist} are satisfied and that the sequence $(X_i)_{i\in\N}$ is $\eta$-weakly dependent.  Then, for any $r$, $\widehat{w}_r$ exists eventually almost surely. 
\end{coro}
\begin{proof}
This is a direct consequence of Proposition \ref{exist_w^i} and Lemma \ref{weak_dep_Z}. 
\end{proof}
\begin{proof}[Proof of Theorem \ref{TCL3}]
From Lemma \ref{weak_dep_Z} and Theorem \ref{TCL2}, we get 
$$\sqrt{k}\left(\widehat{M}_k^r(t) - \E(e^{tZ_1^r})\right) \stackrel{\LL}{\longrightarrow} \NN(0\/, \Gamma^2_r(t))$$
with 
$$\Gamma^2_r(t) = \sum_{n\geq 0} \cov(e^{tZ_0^r}\/, e^{tZ_n^r}) \/$$
and if we denote by $w_r$ the $i$-adjustment coefficient of the sequence $Z_0^r$ and $\widehat{w_r}$ the positive solution to $\widehat{M}_k^r(t)=1$,
$$\sqrt{k}(\widehat{w_r}-w_r) \stackrel{\LL}{\longrightarrow} \NN(0\/, \Gamma^2_r)\/$$
with 
$$\Gamma^2_r = \frac{\Gamma^2_r(w_r)}{\E(Z_1^r e^{w_rZ_1^r})^2} \/.$$
This implies that $\widehat{w_r}$ goes to $w_r$ in probability, as $k$ goes to infinity.  We conclude the proof of Theorem \ref{TCL3} by using Corollary \ref{conv_adjust}~: there exists a sequence $r(k)\stackrel{k\rightarrow \infty}\longrightarrow \infty$ such that $\widehat{w}_r$ converges to $w^d$ in probability. 
\end{proof}
Theorem \ref{TCL3} is interesting from a theoretical point of view but it is not so useful from a practical point of view. Indeed, it proves that $\widehat{w}_r$ converges to $w^d$ for a sequence $r=r(k)$ but we have no information on how to choose $r$ with respect to $k$. Moreover, provided that $\frac{\Gamma^2(w_r)}{\E(Z_1^re^{w_rZ_1^r})^2}$ is converging, we could obtain a central limit theorem for $\sqrt{k}(\widehat{w}_r-w^d)$ with limit variance $\Gamma^2_d  =\displaystyle \lim_{r\rightarrow \infty} \frac{\Gamma^2(w_r)}{\E(Z_1^re^{w_rZ_1^r})^2}$. This expression of the asymptotic variance in the central limit theorem is not useful from a practical point of view. We might nevertheless use moment and Bienaimé-Tchebitchev inequalities in order to get a useful relationship between $r$ and $k$. \\ 
 We apply the following  inequality on the order $2$ moment.   
\begin{prop} \label{expo1}
Let $(W_i)_{i\in\N}$ be a sequence of centered random variables and  
$$C_{j\/,2} = \sup_{\stackrel{t_1\/, \ t_2}{ t_2-t_1=j}}\cov(W_{t_1}\/,W_{t_2}) \/.$$
$$\displaystyle \E(S_n^2) \leq 2n \sum_{j=0}^{n-1}C_{j\/,2} \/,$$
where $\displaystyle S_n=\sum_{i=1}^n W_i$.\\
As a consequence, if $W_i$ is $\eta$-dependent and stationary, with mixing coefficient $\eps(k)\leq C\theta^k$, we have that 
$$\E(S_n^2)\leq 16nM_m^{\frac{2}m}C^{\frac{m-2}m}\frac1{1-\theta^{\frac{m-2}m}}\/$$
where $M_m = \E(|W_i|^m)$.
\end{prop} 
\begin{proof}
The first part of the proof of Proposition \ref{expo1} may be found in \cite{book} (see Lemma 4.6 p.79). For the second part, we proceed as in the proof of Lemma \ref{heredit_eta} (see also \cite{DN})~: let $W_t^{(M)} = \max(\min(W_t\/,M)\/, -M)$, (so that $W_t^{(M)} = W_t$ provided that $|W_t|\leq M$). Then,
\begin{eqnarray*}
\cov(W_t\/,W_{t+j})&\leq &\cov(W_t^{(M)}\/,W_{t+j}^{(M)}) + \cov(W_t^{(M)}\/,(W_{t+j}-W_{t+j}^{(M)}))+ \\
\lefteqn{\cov(W_{t+j}^{(M)}\/,W_j-W_j^{(M)}) + \cov((W_{t+j}-W_{t+j}^{(M)})\/,W_j-W_j^{(M)}) }\\
&\leq&2 M^2\eps(j) + 2M \Vert W_t-W_t^{(M)}\Vert_1 +\Vert ( W_t-W_t^{(M)})^2\Vert_1\\
&\leq &2 M^2\eps(j) + 6M^{2-m}M_m\/,
\end{eqnarray*}
where the last line is obtained by noting that
\begin{eqnarray*}
\Vert W_t-W_t^{(M)}\Vert_1&=&\int \Id_{|W_t|>M}|W_t-W_t^{(M)}|d\P\\
	&\leq & 2 \int \Id_{|W_t|>M}|W_t|d\P \\
	&\leq & 2 M^{1-m}M_m\/.
\end{eqnarray*}
Similarly, we obtain 
$$\Vert ( W_t-W_t^{(M)})^2\Vert_1\leq 4M^{2-m}M_m \/.$$
We conclude by choosing $M= \displaystyle \left(\frac{M_m}{\eps(j)}\right)^{\frac1m}$.
\end{proof}
\begin{prop}\label{expoM}
Assume that $(X_i)_{i\in\N}$ is an $\eta$ weakly-dependent process, with mixing coefficient $\eps(k) \leq C\theta^k$, $C>0$, $0<\theta<1$. Then, for any  $t\in[0\/,u_0[$, such that $3t\leq u_0$, for any $v>0$, we have:
$$\P\left(|\widehat{M_k^r}(t) - \E(e^{tY_r})|>v \right) \leq \frac{4(Cr+3)\E(e^{3tY_r})^{2/3}}{v^2 k (1-\theta^{\frac16})}\/.$$
\end{prop}
\begin{proof}
We apply the Bienaimé-Tchebitchev inequality and Proposition \ref{expo1} with $m=3$. Let 
$$W_\ell(t) = e^{t\sum_{i=1}^r X_{i+r\ell}} \/.$$
Following the lines of the proof of Proposition \ref{expo1}, we have that the covariance coefficients associated to $(W_\ell(t))_{\ell \in\N}$ are:
$$C_{2\/,j} \leq M_m^{\frac2m}\theta^{\frac{j}2\frac{m-2}m}(2Cr+6) \/,$$
with $M_m=\E(e^{tmY_r})$. We have used the fact that
$$\cov(W_\ell^{(M)}(t)\/,W_{\ell+j}^{(M)}(t))\leq 2M_m^2r\eps(r(j-1)+1)\leq 2M_m^2rC \theta^{\frac{j}2}\/.$$
We choose $m=3$ and apply Proposition \ref{expo1}, so that 
$$\E([k\widehat{M_k^r}(t) - k\E(e^{tY_r})]^2)\leq 2k(2Cr+6)[\E(e^{3tY_r})]^\frac23\frac1{1-\theta^\frac16}\/.$$
We conclude by using the  Bienaimé-Tchebitchev inequality.
\end{proof} 
This proposition  shows that in order to get the consistency for $\widehat{w}_r$ we should choose  $r(k) = o(\ln k)$. 
\begin{theo} \label{as_cv}
Assume that  $(X_i)_{i\in\N}$ is an $\eta$ weakly-dependent process, with mixing coefficient $\eps(k) = O(\theta^k)$, $0<\theta<1$ and that $3w^d<u_0$. Then, for $r=r(k)= o(\ln k)$, $\widehat{w}_r$ goes to $w^d$ in probability. 
\end{theo} 
\begin{proof} 
We have that 
\begin{eqnarray*}
\E(e^{w_rZ_1^r}) - \widehat{M}_k^r(w_r)& =& \widehat{M}_k^r(\widehat{w}_r)-\widehat{M}_k^r(w_r)\\
&=&(\widehat{w}_r-w_r)\frac{\partial \widehat{M}_k^r(w_r)}{\partial t} + \frac{\widehat{w}_r-w_r}2 \int\limits_{\widehat{I}_r} \frac{\partial^2 \widehat{M}_k^r(w)}{\partial t^2} \/ d\/w \/,
\end{eqnarray*}
with $\widehat{I}_r=[\min(\widehat{w}_r\/,w_r)\/, \max(\widehat{w}_r\/,w_r)]$, so that 
\begin{eqnarray*}|\widehat{w}_r - w_r| &=& |\E(e^{w_rZ_1^r}) - \widehat{M}_k^r(w_r)|  \left[ \frac{\partial \widehat{M}_k^r(w_r)}{\partial t}  + \frac12 \int\limits_{\widehat{I}_r} \frac{\partial^2 \widehat{M}_k^r(w)}{\partial t^2} \/ d\/w \right]^{-1} \\
&\leq & |\E(e^{w_rZ_1^r}) - \widehat{M}_k^r(w_r)|  \left[ \frac{\partial \widehat{M}_k^r(w_r)}{\partial t}  \right]^{-1} \ \mbox{remark that} \ \frac{\partial^2 \widehat{M}_k^r(w)}{\partial t^2}>0 \/.
\end{eqnarray*}
For any $L>0$,
\begin{eqnarray*}
\P(|\widehat{w}_r -w_r|>u) &=&\P\left(|\E(e^{w_rZ_1^r}) - \widehat{M}_k^r(w_r)|>u\frac{\partial \widehat{M}_k^r(w_r)}{\partial t} \right)\\
&\leq&\P(|\widehat{M}_k^r(w_r) - \E(e^{w_rZ_1^r})|>u L) + \P\left(\left|\frac{\partial \widehat{M}_k^r(w_r)}{\partial t}\right|\leq L\right)\/.
\end{eqnarray*}
Denote $\alpha_r(t) = \E\left(\frac{\partial \widehat{M}_k^r(t)}{\partial t}\displaystyle\right) = \E(Z_1^re^{tZ_1^r}) = \frac{\partial}{\partial t} \E(e^{tY_r})$ and take $L=\frac{\alpha_r(w_r)}2$. Note that $\alpha_r(w_r)>0$ because of the convexity of $\E(e^{tZ_1^r})$ and the fact that $w_r$ is the unique positive solution to $\E(e^{tZ_1^r})=1$. Hence, 
\begin{eqnarray*}
\P\left(\left|\frac{\partial \widehat{M}_k^r(w_r)}{\partial t}\right|\leq L\right)&\leq & \P\left(\left|\frac{\partial \widehat{M}_k^r(w_r)}{\partial t}-\alpha_r(w_r)\right|>\frac{\alpha_r(w_r)}2\right)\\
&\leq& \frac{16(Cr+3)\E(Y_r^3e^{3w_rY_r})^\frac23}{k\alpha_r(w_r)^2(1-\theta^{\frac16})}\/
\end{eqnarray*}
(we proceed as in the proof of Proposition \ref{expoM}). Finally,
\begin{equation}\label{expolast}
\P(|\widehat{w}_r -w_r|>v)\leq  \frac{16(Cr+3)}{\alpha_r(w_r)^2k(1-\theta^\frac16)}(\E(e^{3w_rY_r})^\frac23\frac1{v^2}+\E(Y_r^3e^{3w_rY_r})^\frac23)\/.
\end{equation}
Remark that $\alpha_r(w_r) = r E(e^{w_rY_r})c_r^\prime(w_r)$. Consider an interval $[u_1\/,u_2]\subset [0\/,u_0[$ such that $w^d\in [u_1\/,u_2]$, $3u_2<u_0$, $c$ is non decreasing on $ [u_1\/,u_2]$, $c(u_1)<0$, $c(u_2)>0$ (this exists because of the convexity of the function $c$). Since $c_r$ converges uniformly to $c$ and   $w_r$ converges to $w^d$,  for $r$ large enough, $w_r \in [u_1\/,u_2]$, and 
$$E(e^{3w_rY_r}) \leq e^{r(\eps + c(3u_2))} \/.$$
We also have that for $r$ large enough,  $\alpha_r(w_r) \geq \displaystyle \frac{1-e^{r(c(u_1)-\eps)}}{u_2-u_1}$. By taking  $r(k)=o(\ln k)$, we have that $\widehat{w}_r$ goes to $w^d$ in probability.
\end{proof}
\section{Simulations}\label{simul}
We conclude by giving some simulation results. We present some models for which the adjustment coefficient is computable - namely MA and AR linear processes with an innovation following an exponential law. We provide also a non linear example. We refer to \cite{CMM} for non linear and computable examples. These examples are also more realistic from an actuarial point of view. \\
Recall that if $\xi_i$ follows an exponential law with parameter $\theta >0$ then for $0\leq t<\theta$,
$$\E(e^{t\xi_i})= \frac{\theta}{\theta-t}\/.$$
Let $\eps_i = \xi_i-c$, with $c\theta >1$.  Then, the independent adjustment coefficient $w^i$ is the positive solution to~:
$$e^{-tc}\frac\theta{\theta-t} = 1 \/.$$
The simulation results are summarized below. The graphs represent the estimator $\log \widehat{m}_k(t)$ and $\frac1r\log \widehat{M}_k^r(t)$ of $\lambda(t)$ and $c_r(t)$ respectively in grey and in black. 
\subsection{Independent case}
We have simulated an independent sample of $\eps_i = \xi_i-c$ of length $10000$ and $\theta=1.2$, $c=1$. We have computed $\widehat{w}^i$ and $\widehat{w}^d$:
\begin{center}
\begin{tabular}{|c|c|c|c|}
\hline
$r=6$&$w^i=w^d=0.38$&$\widehat{w}^i=0.36$&$\widehat{w}^d=0.37$\\
\hline
\end{tabular}\\
\includegraphics[scale=0.4,angle=-90]{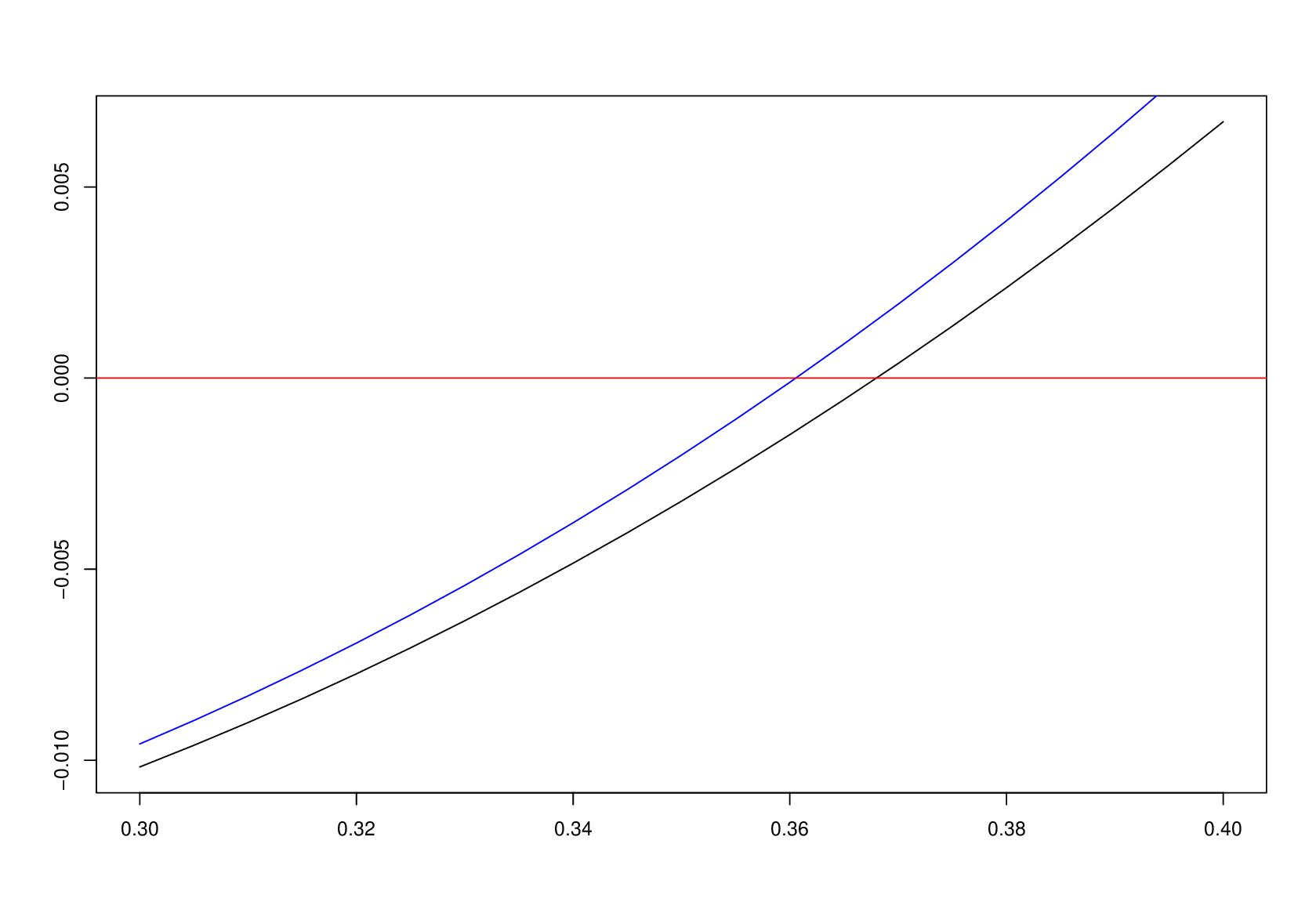}
\end{center}
\subsection{$AR(1)$ model}
We consider the following  $AR(1)$ model: $X_n = aX_{n-1}+\eps_n$. Following \cite{G2}, we have that $w^d = (1-a)w^i$. We have simulated a sample of length $10000$ for $\theta=1.2$, $c=1$, $a=0.3$. Then,
\begin{center}
\begin{tabular}{|c|c|c|c|}
\hline
$r=6$&$\widehat{w}^i=0.45$&$w^d=0.26$&$\widehat{w}^d=0.27$\\
\hline
\end{tabular}\\
\includegraphics[scale=0.4,angle=-90]{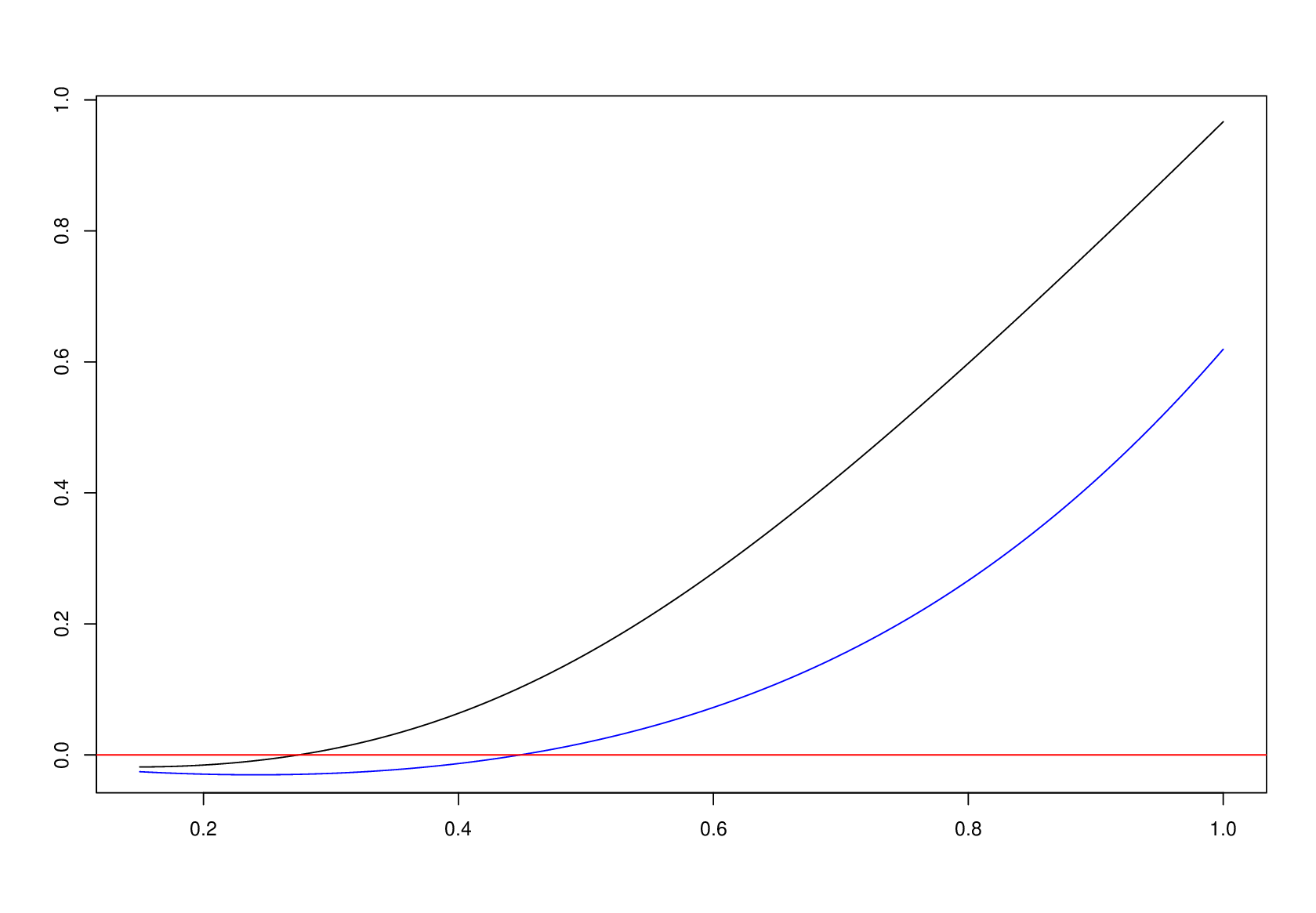}
\end{center}
\subsection{$MA(1)$ model} \label{sec_ma}
We consider the following $MA(1)$ model: $X_n = \eps_n + a\eps_{n-1}$, with $\theta=1.2$, $c=1$, $a=0.2$. Then, $w^d$ is the positive solution to:
$$-tc(1+a)+\ln \theta - \ln(\theta-t(1+a))=0 \/.$$
We have simulated a sample of length $10000$,
\begin{center}\begin{tabular}{|c|c|c|c|c|}
\hline
$r=6$&$\widehat{w}^i=0.47$&$w^d=0.31$&$\widehat{w}^d=0.32$\\
\hline
\end{tabular}\\
\includegraphics[scale=0.4,angle=-90]{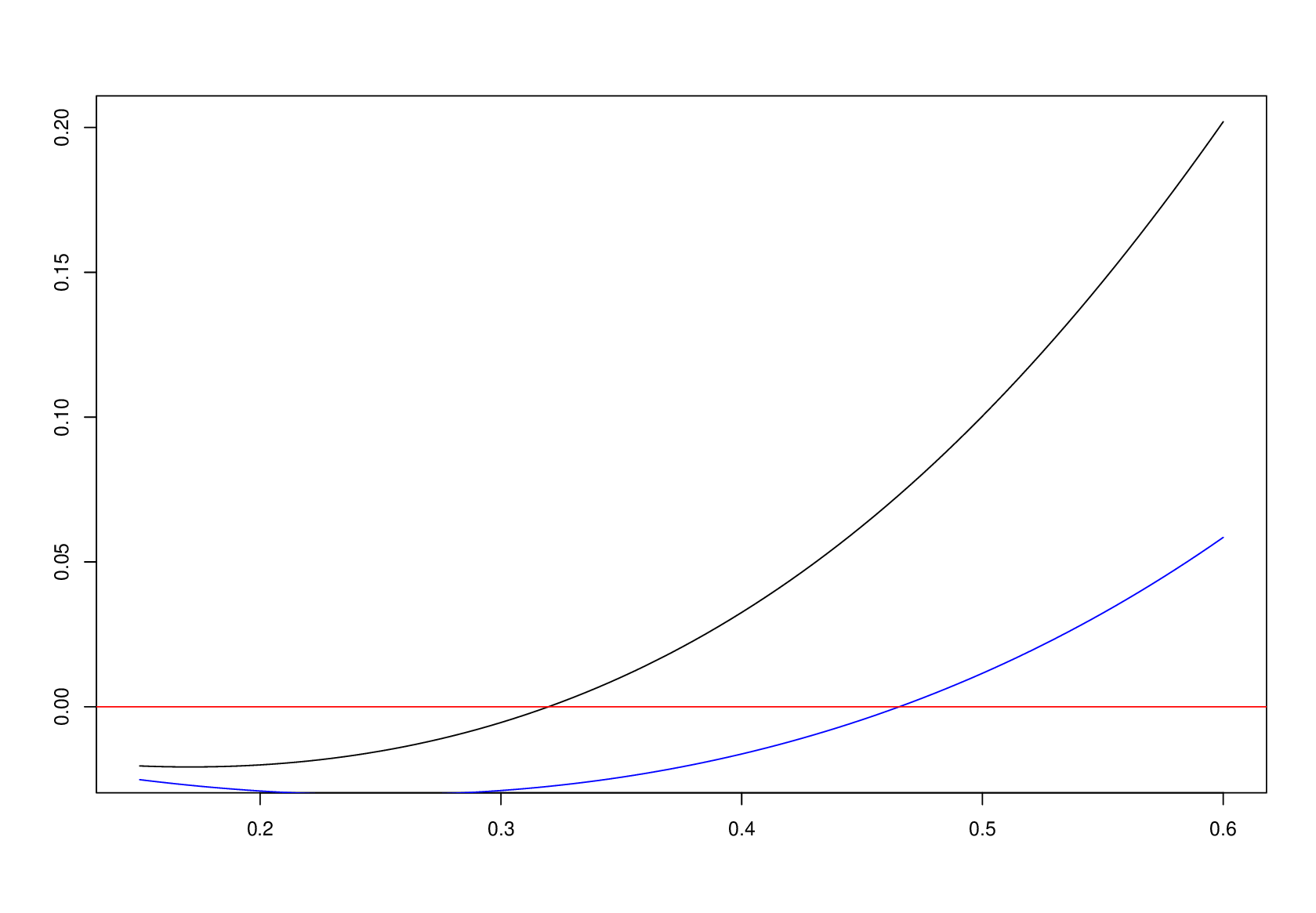}
\end{center}
\subsection{A non linear $AR(1)$ model}\label{nonlin}
We consider the following non linear $AR(1)$ model (which may be seen as a particular case of Bernoulli shifts, see \cite{book}): $X_n = aX_{n-1}^2 + 0.7\eps_n$. We have simulated a sample of length $10000$, with $\theta=1.2$, $c=1$, $a=-0.2$.
\begin{center}\begin{tabular}{|c|c|c|c|}
\hline
$r=6$&$\widehat{w}^i=0.8$&$\widehat{w}^d=1.21$\\
\hline
\end{tabular}\\
\includegraphics[scale=0.4,angle=-90]{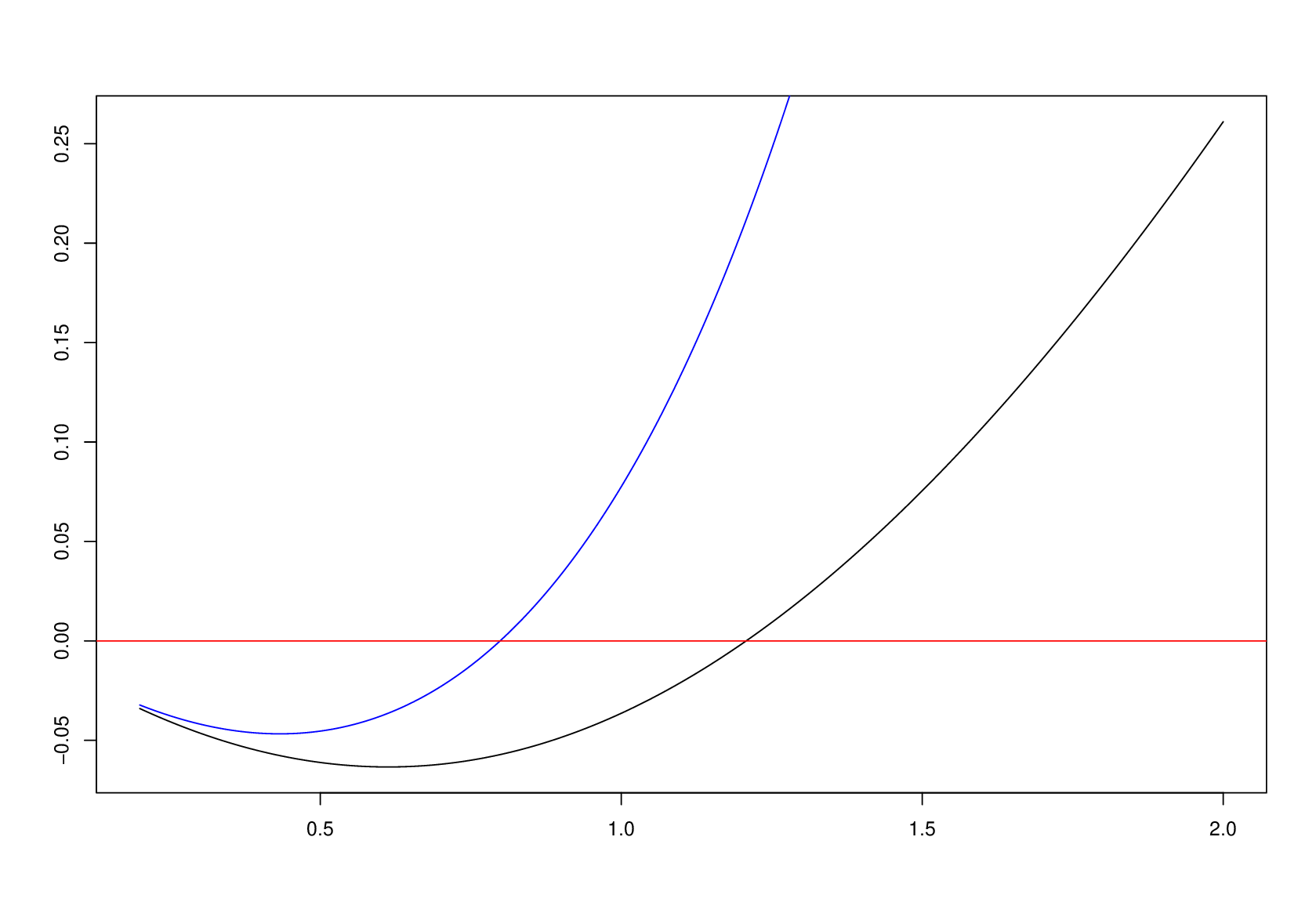}
\end{center}
\subsection{How to choose the $r$-parameter~?} When performing the estimation of the $w^d$ coefficient, we are faced with the choice of the parameter $r$. Following Theorem \ref{as_cv}, we should take $r=o(\ln(k))$ but the practical choice of $r$ for $n$ given is not clear. We have performed several simulations for the independent, $MA(1)$, $AR(1)$, non linear $AR(1)$ models, for several values of $r$. These experiments tend to show that when $r$ increases, the estimator $\widehat{w}^d$ behaves monotonically in the beginning and then has a more chaotic behavior. We propose to choose $r$ as the largest integer for which $\widehat{w}^d$ is monotonic on $[0\/,r]$. This is illustrated in the graphs below for several models.
\subsubsection{Independent case} We have simulated an independent sample of $\eps_i = \xi_i-c$ of length $10000$ and $\theta=1.2$, $c=1$. Below is represented $\widehat{w}^d$ for $r=1\/, \ldots \/, 35$, $w^i=w^d=0.38$. 
\begin{center}
\includegraphics[scale=0.4,angle=-90]{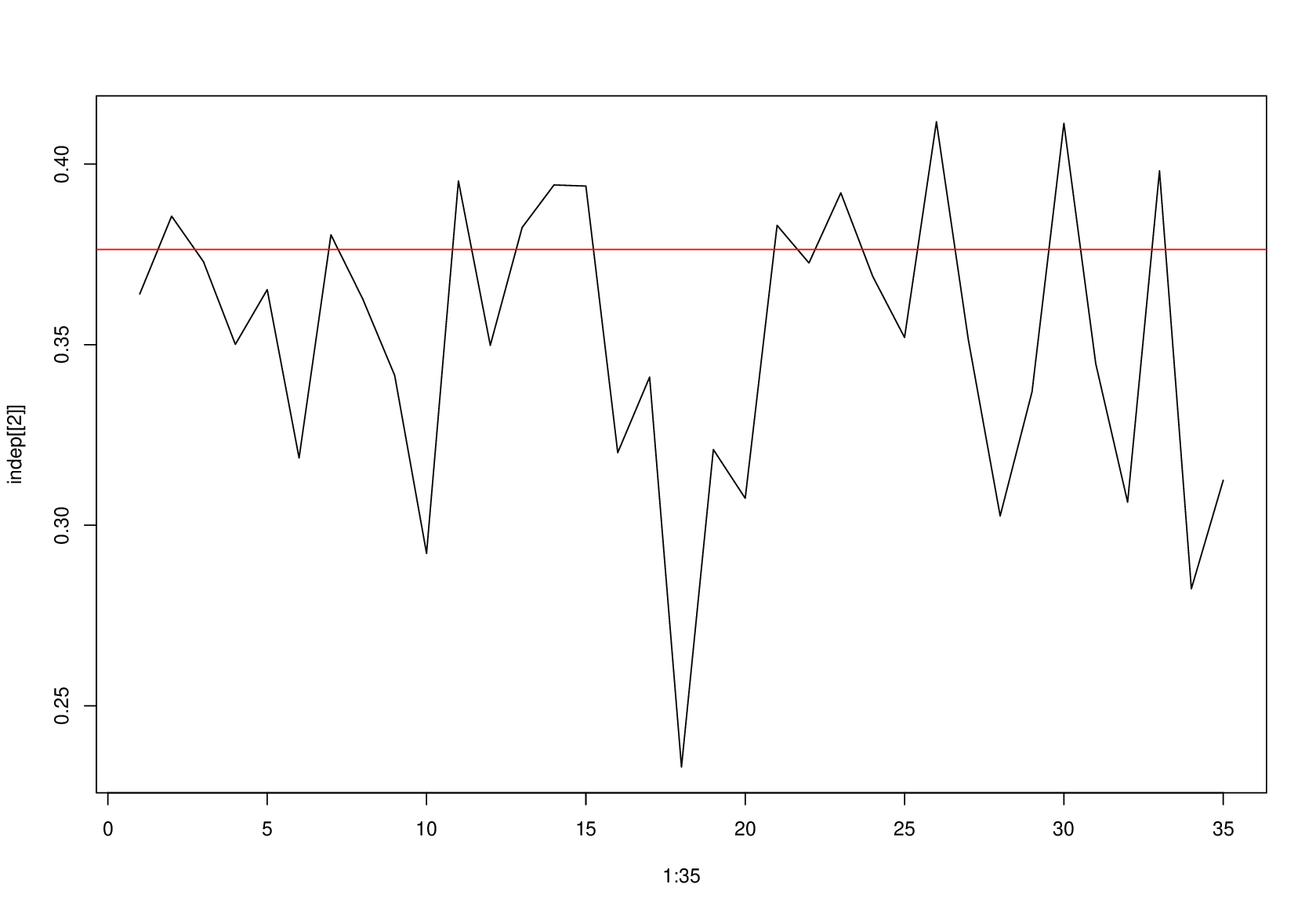}
\end{center}

\subsubsection{Linear $MA(1)$} We consider the following $MA(1)$ model: $X_n = \eps_n + a\eps_{n-1}$, with $\theta=1.2$, $c=1$, $a=0.3$. We have simulated a sample of size $10000$ and represented below $\widehat{w}^d$ for $r=1\/, \ldots \/, 40$, $w^d=0.26$. 
\begin{center}
\includegraphics[scale=0.4,angle=-90]{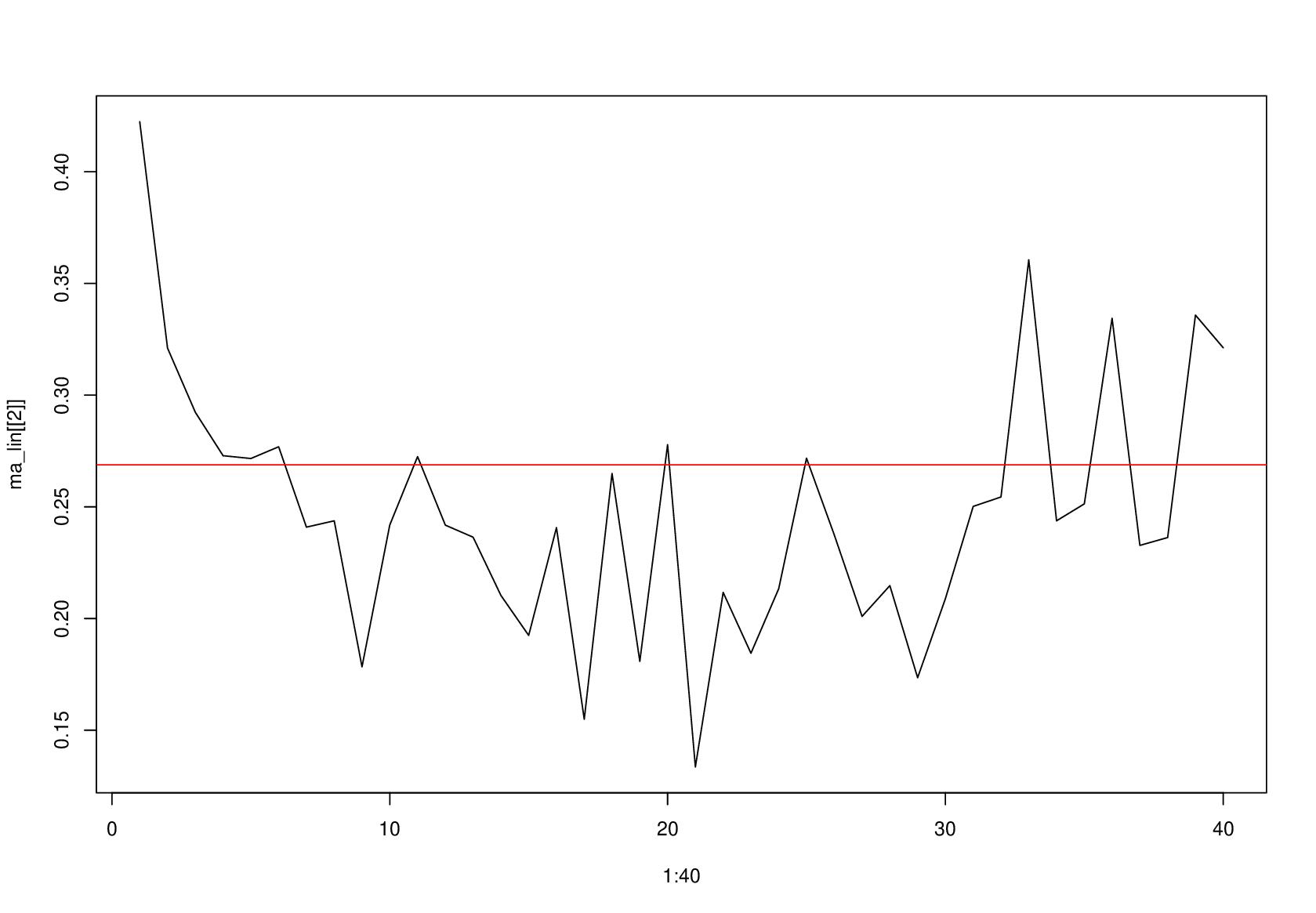}
\end{center}
\subsubsection{Linear $AR(1)$} We consider the following  $AR(1)$ model: $X_n = aX_{n-1}+\eps_n$. We have simulated a sample of length $10000$ for $\theta=1.2$, $c=1$, $a=0.4$, $w^d=$.
\begin{center}
\includegraphics[scale=0.4,angle=-90]{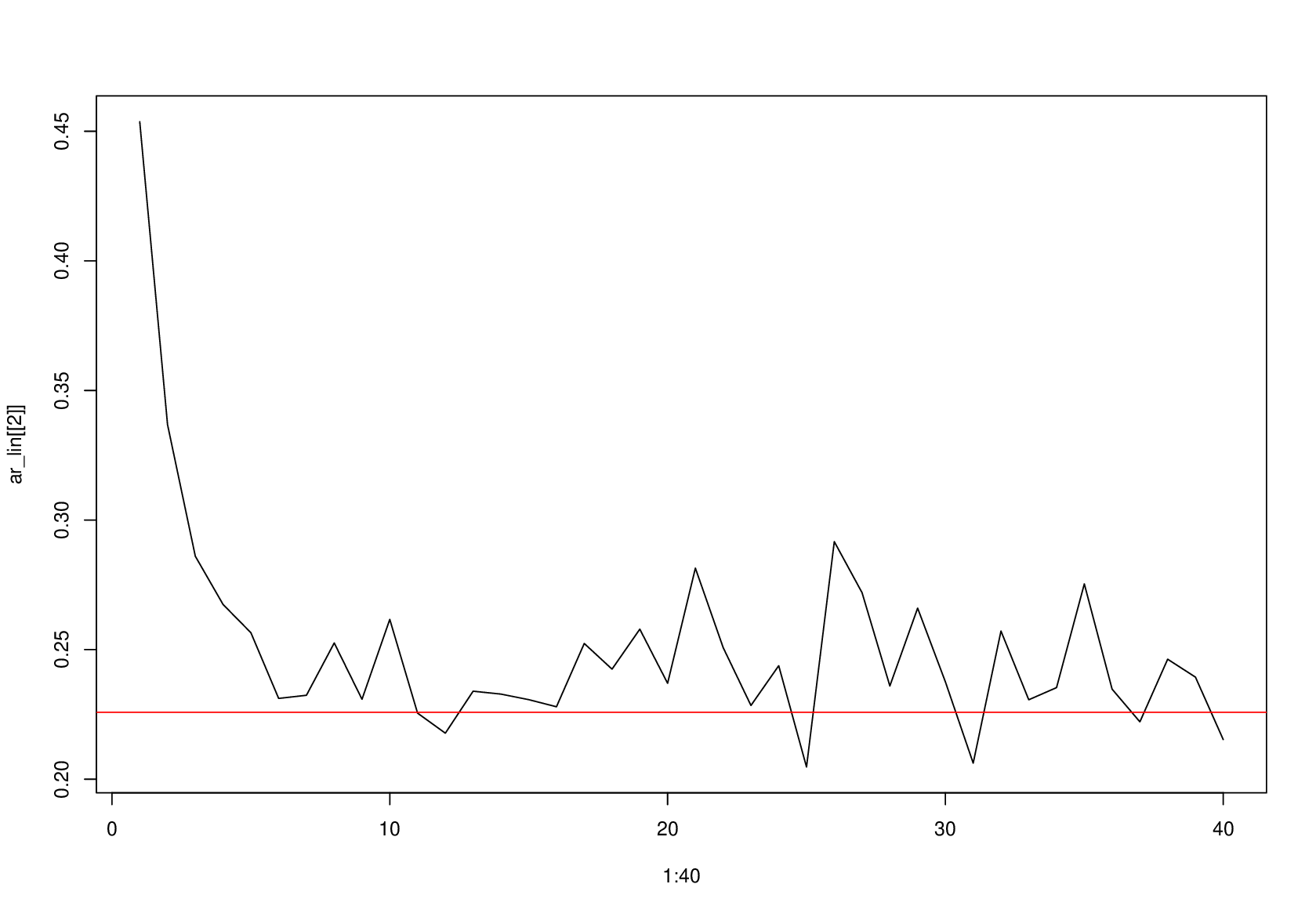}
\end{center}
\subsection{Non linear $AR(1)$} We have represented below $\widehat{w}^d$ for $r=1\/, \ldots\/, 35$ for the non linear $AR(1)$ model of section \ref{nonlin}. 
\begin{center}
\includegraphics[scale=0.4,angle=-90]{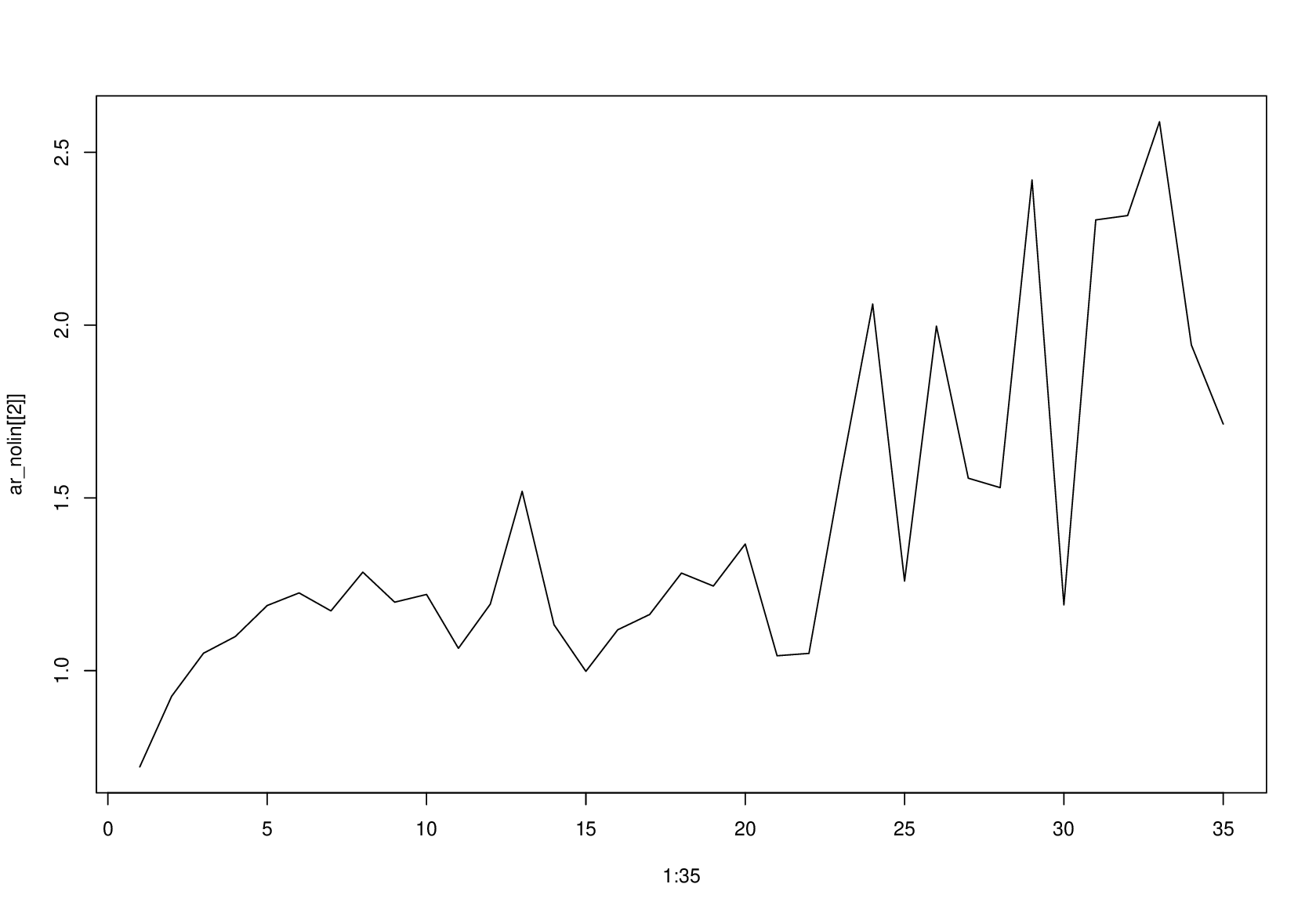}
\end{center}

\subsection{On the empirical distribution of $\widehat{w}^d$.}
We conclude this simulation section with a short study of the empirical distribution of $\widehat{w}^d$. We have performed $100$ simulations of a sample of size $10000$ in the $MA(1)$ model (Section \ref{sec_ma}). The mean value of $\widehat{w}^d$ is $0.317 $, with standard deviation $0.04 $. The computed value of $w^d$ is $0.314$. The histogram and a Shapiro test indicate that the distribution of $\widehat{w}^d$ is probably asymptotically normal.\ \\
{\tt  \ \ \	Shapiro-Wilk normality test\\
\ \
W = 0.9871, p-value = 0.4462
} 
\begin{center}
\includegraphics[scale=0.4,angle=-90]{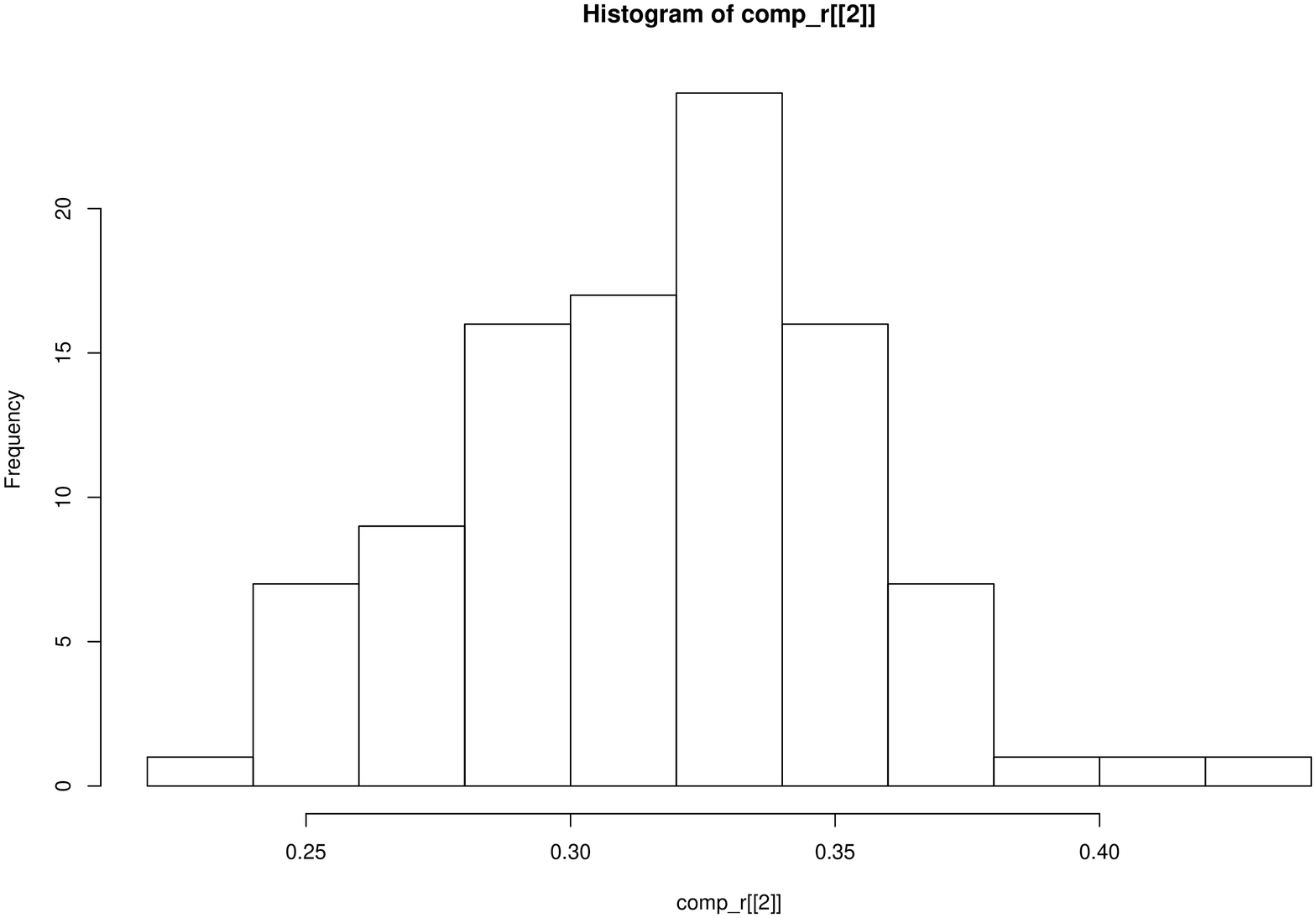}
\end{center}


\end{document}